\documentclass[12pt]{amsart}
\usepackage{amssymb}
\usepackage[margin=1in]{geometry}
\usepackage[colorlinks]{hyperref}
\newtheorem{theorem}{Theorem}[section]
\newtheorem{lem}[theorem]{Lemma}

\newtheorem{cor}[theorem]{Corollary}

\theoremstyle{definition}

\newtheorem{example}[theorem]{Example}

\theoremstyle{remark}
\newtheorem{remark}[theorem]{Remark}

\numberwithin{equation}{section}
\begin{document}

\newcommand{\spacing}[1]{\renewcommand{\baselinestretch}{#1}\large\normalsize}
\spacing{1.14}

\title[Invariant Einstein Kropina metrics]{Invariant Einstein Kropina metrics on Lie groups and homogeneous spaces }

\author {M. Hosseini}

\address{Masoumeh Hosseini\\ Department of Pure Mathematics \\ Faculty of  Mathematics and Statistics\\ University of Isfahan\\ Isfahan\\ 81746-73441-Iran.} \email{hoseini\_masomeh@ymail.com}

\author {H. R. Salimi Moghaddam}

\address{Hamid Reza Salimi Moghaddam\\ Department of Pure Mathematics \\ Faculty of  Mathematics and Statistics\\ University of Isfahan\\ Isfahan\\ 81746-73441-Iran.\\ Scopus Author ID: 26534920800 \\ ORCID Id:0000-0001-6112-4259\\} \email{hr.salimi@sci.ui.ac.ir and salimi.moghaddam@gmail.com}



\begin{abstract}
    In this article, we study Einstein Kropina metrics on Lie groups and homogeneous spaces. We give a method to construct  Einstein Kropina metrics on Lie groups. As an example of this method, a family of non-Riemannian Einstein Kropina metrics on the special orthogonal group $SO(n)$ is given. Then, we classify all left invariant Einstein Kropina metrics  on simply connected $3$-dimensional real Lie groups. We provide a procedure to build Einstein Kropina metrics on homogeneous spaces. Using this technique, we study invariant Einstein Kropina metrics on spheres. Finally, we show that  projective spaces do not admit any homogeneous  non-Riemannian  Einstein Kropina metrics.
\end{abstract}

\maketitle

\section{\textbf{Introduction}}

The family of $(\alpha , \beta)$-metrics is a special class of Finsler metrics. They are very interesting not only for their simplicity, but also for their applications (see \cite{Antonelli-Ingarden-Matsumoto, Asanov}).
For a Riemannian metric $\textbf{\textsf{a}}$ and a 1-form $\beta$ on a smooth manifold $M$, an $(\alpha , \beta)$-metric is a function defined as $F = \alpha \phi (\frac{\beta}{\alpha})$, where $\alpha (x,y)=\sqrt{\textbf{\textsf{a}}(y,y)}$ and $\phi : (-b_0 , b_0) \longrightarrow \mathbb{R^+}$ is a $C^\infty$ function. It is well known that $F$ is a Finsler metric if
\begin{equation}
\phi (s) - s \phi ^{'} (s) + ( b^2 - s^2 ) \phi ^{''} (s) > 0, \qquad  \vert s \vert  \leq b < b_0,
\end{equation}
and  $\|\beta\|_\alpha < b_0$ (see  \cite{Chern-Shen}).\\
Two important types of $(\alpha , \beta)$-metrics happen when $\phi(s) = 1+s$ and $\phi (s) = \frac{1}{s}$. In these cases we have $F=\alpha + \beta$ and $F=\frac{\alpha ^2}{\beta}$ which are called Randers metric and Kropina metric respectively. Although the Randers metric is a regular Finsler metric but the Kropina metric is singular. In this article we consider the Kropina metrics on the domain $\{(x,y)\in TM|\beta(y)>0)\}$ (see \cite{Xia}).

Another description of Kropina metrics is the definition of these metrics as the solutions of the navigation problem on some Riemannian manifold $(M,\textbf{\textsf{h}})$ under the influence of a vector field $W$ with $\|W\|_{\textbf{\textsf{h}}}=1$. In this case the pair $(\textbf{\textsf{h}},W)$ is called the navigation data of $F$. In fact $F=\frac{\alpha ^2}{\beta}$ is a Kropina metric on a manifold $M$ if and only if $F=\frac{{\textbf{\textsf{h}}}^2}{2W_0}$, where ${\textbf{\textsf{h}}}^2=e^{2\rho}\alpha^2$, $2W_0=e^{2\rho}\beta$, $e^{2\rho}b^2=4$ and $b=\|\beta\|_{\textbf{\textsf{h}}}$, for some functions $\rho=\rho(x)$ on $M$ (see \cite{Zhang-Shen}).\\
An interesting case of Finsler metrics appears when we consider a left invariant Finsler metric on a Lie group $G$. This is the case that for any $x\in G$ and $y\in T_xG$ we have
\begin{equation}
F(x,y)=F(e,dl\,_{x^{-1}}y)\ \ \text{for every}\ x\in G, y\in T_x G,
\end{equation}
where $l_x$ denotes the left translation and $e$ denotes the unit element of $G$.\\
A very useful way for denoting left invariant $(\alpha , \beta)$-metrics, is the use of left invariant vector fields instead of 1-forms $\beta$, namely
\begin{equation}
\textbf{\textsf{a}}(y,X(x)) = \beta (x,y),
\end{equation}
where $\textbf{\textsf{a}}$ is a left invariant Riemannian metric and $X$ is a left invariant vector field on $G$ such that $\|X\|_\alpha=\textbf{\textsf{a}}(X,X) < b_0$.
In this way, for a left invariant vector field $X$ and a left invariant Riemannian metric $\textbf{\textsf{a}}$, we can represent a left invariant Kropina metric as follows
\begin{equation}\label{left invariant Kropina metric}
    F(x,y)=\frac{\textbf{\textsf{a}}(y,y)}{\textbf{\textsf{a}}(X(x),y)}.
\end{equation}
Suppose that $(M,F)$ is an $n$-dimensional Finsler manifold. Let $\sigma$ be a scalar function on the manifold $M$. If the Finsler metric $F$ satisfies the equation
\begin{equation}\label{Ricci Equation}
    \textsf{Ric}_F=\sigma F^2,
\end{equation}
then $F$ is named an Einstein metric with Einstein scalar $\sigma$. If $F$ satisfies the equation (\ref{Ricci Equation}) with a constant function $\sigma=c$ then $F$ is called Ricci constant. In particular case if $\sigma=0$ then $F$ is named Ricci flat (see \cite{Chern-Shen}, \cite{Xia} and \cite{Zhang-Shen}).\\
Throughout this paper we consider that $G$ is an $n$-dimensional connected real Lie group such that $n\geq2$.


\section{left invariant Einstein Kropina metrics on Lie groups}
In this section we give some ways for constructing Einstein Kropina metrics on Lie groups.
\begin{theorem}\label{Theorem 1}
Let $G$ be an $n$-dimensional Lie group equipped with an Einstein left invariant Riemannian metric $\textbf{\textsf{h}}$ with the Einstein scalar $\sigma$. Suppose that $W$ is a right invariant vector field on $G$ such that $\|W\|_{\textbf{\textsf{h}}}=1$. Then the corresponding Kropina metric to the navigation data $(\textbf{\textsf{h}}, W)$ is an Einstein non-Riemannian  Finsler metric with the same Einstein scalar $\sigma$. Moreover if $n\geq3$ then the induced Kropina metric is Ricci constant.
\end{theorem}
\begin{proof}
Assume that $F=\frac{\alpha^2}{\beta}$ is a non-Riemannian Kropina metric on a manifold $M$ such that $\dim(M)\geq 2$. Suppose that $(\textbf{\textsf{h}}, W)$ is the navigation data of $F$. In \cite{Zhang-Shen}, it is shown that $F$ is an Einstein metric if and only if $\textbf{\textsf{h}}$ is an Einstein metric and $W$ is a Killing vector field of the Riemannian manifold $(M,\textbf{\textsf{h}})$ such that $\|W\|_{\textbf{\textsf{h}}}=1$.
In this case, $F$ and $\textbf{\textsf{h}}$ have the same Einstein scalar $\sigma=\sigma(x)$. Also if $\dim(M)\geq 3$ then $F$ is Ricci constant (see Theorem 4.2 of \cite{Zhang-Shen}). Now it is sufficient to mention that, for any left invariant metric, the right invariant vector fields are
Killing vector fields (see \cite{Torres-Book}).
\end{proof}
The above theorem yields the following useful corollaries.
\begin{cor}\label{Corollary 1}
Assume that $G$ is an $n$-dimensional Lie group with an Einstein left invariant Riemannian metric $\textbf{\textsf{h}}$, and $\sigma$ denotes it's Einstein scalar. If the non-zero vector field $W$ commutes with all left invariant vector fields and $\|W\|_{\textbf{\textsf{h}}}=1$ then the corresponding Kropina metric to the navigation data $(\textbf{\textsf{h}}, W)$ is an Einstein non-Riemannian  Finsler metric with the same Einstein scalar $\sigma$ and for $n\geq3$ the induced Kropina metric is Ricci constant.
\end{cor}
\begin{proof}
It is well known that, on a connected Lie group $G$, if a vector field $W$ commutes with all left invariant vector fields then it is right invariant (for more details see Theorem 7.50 of \cite{Torres-Book}). Now it is sufficient to use the previous theorem.
\end{proof}
\begin{cor}\label{Corollary 2}
Suppose that $G$ is an $n$-dimensional Lie group equipped with an Einstein left invariant Riemannian metric $\textbf{\textsf{h}}$ with the Einstein scalar $\sigma$. Assume that $\frak{g}$ denotes the Lie algebra of left invariant vector fields of $G$ and $W'\in Z(\frak{g})$ is a non-zero left invariant vector field, where $Z(\frak{g})$ denotes the center of $\frak{g}$. Then the corresponding Kropina metric to the navigation data $(\textbf{\textsf{h}}, W=\frac{W'}{\|W'\|_\textbf{\textsf{h}}})$ is an Einstein non-Riemannian Finsler metric with the Einstein scalar $\sigma$. In this case, if $n\geq3$ then the corresponding Kropina metric is Ricci constant.
\end{cor}
The following theorem gives us a reach family of Einstein Kropina metrics.
\begin{theorem}\label{Theorem 2}
Let $G$ be a compact semi-simple Lie group furnished with the bi-invariant Riemannian metric $\textbf{\textsf{h}}=-B$ where $B$ denotes the Killing form of $G$. Suppose that $W'$ is an arbitrary non-zero left invariant vector field (right invariant vector field). Then the corresponding Kropina metric to the navigation data $(\textbf{\textsf{h}}, W=\frac{W'}{\|W'\|_\textbf{\textsf{h}}})$ is an Einstein non-Riemannian  Finsler metric with the Einstein scalar $\sigma=\frac{1}{4}$. Also, if $n\geq3$ then the corresponding Kropina metric is Ricci constant.
\end{theorem}
\begin{proof}
Clearly the Riemannian metric $\textbf{\textsf{h}}=-B$ is an Einstein metric with the Einstein scalar $\sigma=\frac{1}{4}$ (see \cite{Arvanitoyeorgos}). On the other hand, for a bi-invariant Riemannian metric on a Lie group, the left invariant and the right invariant
vector fields are Killing vector fields. So the vector field $W=\frac{W'}{\|W'\|_\textbf{\textsf{h}}}$ is a Killing vector field with $\|W\|_\textbf{\textsf{h}}=1$. Now Theorem \ref{Theorem 1} completes the proof.
\end{proof}

\begin{remark}
Suppose that $\textbf{\textsf{h}}$ is an Einstein bi-invariant Riemannian metric on a Lie group $G$ with Einstein scalar $\sigma$ and  $W'$ is an arbitrary non-zero left invariant vector field (right invariant vector field). Then the corresponding Kropina metric to the navigation data $(\textbf{\textsf{h}}, W=\frac{W'}{\|W'\|_\textbf{\textsf{h}}})$ is a non-Riemannian Einstein Kropina metric  with the same Einstein scalar $\sigma$. Furthermore if dim$G\geq3$ then the induced Kropina metric is Ricci constant.
\end{remark}

\begin{example}
Let $\mathfrak{so(n)}$ be the Lie algebra of the special orthogonal group $SO(n)$, $n\geq3$. We consider a basis for $\mathfrak{so(n)}$, consisting of the $n\times n$ matrices $E_{ij}$, ($1\leq i < j\leq n$) that have $1$ in the $(i,j)$ entry, $-1$ in $(j,i)$ entry and $0$ elswhere. It can be seen that the Killing form of $SO(n)$ is given by $B(X,Y)=(n-2)trXY$ for $X,Y \in \mathfrak{so(n)}$ (see \cite{Arvanitoyeorgos}). Applying the previous theorem to the Lie group $G=SO(n)$ and $W=\dfrac{1}{\sqrt{2(n-2)}}E_{ij}$, we obtain a non-Riemannian  Einstein Kropina metric on $SO(n)$ with respect to the navigation data $(\textbf{\textsf{h}}=-B,W)$ as follows
\begin{center}
 $F(Y)=\dfrac{-B(Y,Y)}{-2B(W,Y)}={\sqrt{2(n-2)}}\dfrac{trY^2}{2trYE_{ij}} = \dfrac{\sqrt{2(n-2)}}{2(y_{ji}-y_{ij})} \sum_{i,j=1}^ny_{ij}y_{ji}$,
\end{center}
such that it is Ricci constant. Notice that in the above formula $Y=(y_{ij})$.
\end{example}
Applying  lemma $4$ of \cite{Yoshikawa-Sabau} yields the following lemma:
\begin{lem}
Let $F$ be a Kropina metric with the navigation data $(\textbf{\textsf{h}},W)$ on a homogeneous space $G/H$. Then, $F$ is  $G$-invariant if and only if the Riemannian metric $\textbf{\textsf{h}}$ and the vector field $W$ are $G$-invariant.
\end{lem}
\begin{remark}
Let $F$ be a Kropina metric  with the navigation data $(\textbf{\textsf{h}},W)$ on  a Lie group $G$. Then $F$ is left invariant if and only if the Riemannian metric $\textbf{\textsf{h}}$ and the vector field $W$ are left invariant.
\end{remark}
Now, we classify all left invariant non-Riemannian Einstein Kropina metrics on simply connected $3$-dimensional real Lie groups.
\begin{theorem}\label{Theorem 3}
Suppose that $G$ is a simply connected $3$-dimensional real Lie group and $\mathfrak{g}$ denotes the Lie algebra of $G$. Let $F$ be a left invariant non-Riemannian Kropina metric with the navigation data $(\textbf{\textsf{h}},W)$ on $G$. Then $F$ is an Einstein metric if and only if $G$, $\mathfrak{g}$ and the navigation representation $(\textbf{\textsf{h}},W)$ with respect to the base $\{x,y,z\}$ of $\mathfrak{g}$ up to an automorphism, be as follows:

\begin{itemize}
\item[(i)] $G=\mathbb{R}^3$, $[x,y]=0$, $[x,z]=0$, $[y,z]=0$, $W=ax+by+cz$ where $a^2+b^2+c^2=1$,

$(\textbf{\textsf{h}}_{ij})=\left(
\begin{array}{ccc}
1 & 0 & 0 \\
0 & 1 & 0 \\
0 & 0 & 1
\end{array} \right)$,

\item[(ii)] $G=$The solvable Lie group $\tilde{E_0}(2)$, $[x,y]=0$, $[x,z]=y$, $[y,z]=-x$, $W=\dfrac{1}{\sqrt{\nu}}z$,

$(\textbf{\textsf{h}}_{ij})=\left(
\begin{array}{ccc}
1 & 0 & 0 \\
0 & 1 & 0 \\
0 & 0 & \nu
\end{array} \right)$,
where $0 < \nu $,
\item[(iii)]
$G=$The simple Lie group $SU(2)$, $[x,y]=z$, $[x,z]=-y$, $[y,z]=x$, $W=ax+by+cz$ where $a^2+b^2+c^2=\frac{1}{\lambda}$,

$(\textbf{\textsf{h}}_{ij})=\left(
\begin{array}{ccc}
\lambda & 0 & 0 \\
0 & \lambda & 0 \\
0 & 0 & \lambda
\end{array} \right)$.
\end{itemize}
\end{theorem}
\begin{proof}
In \cite{Zhang-Shen}, it was shown that any Kropina metric $F$ on a manifold $M$ ($dimM \geq 2$) with the navigation data $(\textbf{\textsf{h}},W)$ is an Einstein metric if and only if $\textbf{\textsf{h}}$ is an Einstein metric and $W$ is a Killing vector field of $(M,\textbf{\textsf{h}})$ such that $\Vert W\Vert _{\textbf{\textsf{h}}}=1$.
 In \cite{Salimi Moghaddam}, it is shown that the left invariant Einstein Riemannian metrics on simply connected 3-dimensional real Lie groups are as table \ref{Einstein Riemannian metrics}.
\fontsize{10}{0}{\selectfont
\begin{table}
    \centering\caption{Left invariant Einstein Riemannian metrics on simply connected 3-dimensional real Lie groups}\label{Einstein Riemannian metrics}
        \begin{tabular}{|p{1cm}|p{3.5cm}|p{3.5cm}|p{3cm}|p{2.5cm}|}
        \hline
             Cases & The simply connected  Lie group & Lie algebra structure &  Left invariant Riemannian metric & Conditions for  parameters \\
             \hline
            case 1 & $\Bbb{R}^3$ & $[x,y]=0,$ $[x,z]=0,$ $[y,z]=0$ & $\left(
                                                                         \begin{array}{ccc}
                                                                           1 & 0 & 0 \\
                                                                           0 & 1 & 0 \\
                                                                           0 & 0 & 1 \\
                                                                         \end{array}
                                                                       \right)$ & -  \\
            \hline
            case 2 & The solvable Lie group $\tilde{E}_0(2)$ & $[x,y]=0,$ $[x,z]=y,$ $[y,z]=-x$ &  $\left(
                                                                         \begin{array}{ccc}
                                                                           1 & 0 & 0 \\
                                                                           0 & 1 & 0 \\
                                                                           0 & 0 & \nu \\
                                                                         \end{array}
                                                                       \right)$ & $\nu>0$ \\
            \hline
            case 3 & The simple Lie group $SU(2)$ & $[x,y]=z,$ $[x,z]=-y,$ $[y,z]=x$ & $\left(
                                                                         \begin{array}{ccc}
                                                                           \lambda & 0 & 0 \\
                                                                           0 & \lambda & 0 \\
                                                                           0 & 0 & \lambda \\
                                                                         \end{array}
                                                                       \right)$ & $\lambda>0$ \\
            \hline
            case 4 & The non-unimodular Lie group $G_I$ & $[x,y]=0,$ $[x,z]=-x,$ $[y,z]=-y$ &  $\left(
                                                                         \begin{array}{ccc}
                                                                           1 & 0 & 0 \\
                                                                           0 & 1 & 0 \\
                                                                           0 & 0 & \nu \\
                                                                         \end{array}
                                                                       \right)$ & $\nu>0$ \\
            \hline
             case 5 & The non-unimodular Lie group $G_c$ & $[x,y]=0,$ $[x,z]=-y,$ $[y,z]=cx-2y$ &  $\left(
                                                                         \begin{array}{ccc}
                                                                           1 & 1 & 0 \\
                                                                           1 & c & 0 \\
                                                                           0 & 0 & \nu \\
                                                                         \end{array}
                                                                       \right)$ & $\nu>0 ,$ $c>1$ \\
        \hline
        \end{tabular}
        \end{table}
\fontsize{12}{0}

Hence it is sufficient to find the cases which admit a left invariant Killing vector field $W$ among these $5$ cases. It can be shown that a left invariant vector field $W$ is a Killing vector field if and only if ${\textbf{\textsf{ad}}}W +{\textbf{\textsf{ad}}}^*W=0$. We now observe that in two cases $1$ and $3$ every left invariant vector field $W$ is a Killing vector field. So for every left invariant vector field $W$ such that $\Vert W\Vert _{\textbf{\textsf{h}}}=1$, we have a left invariant Einstein Kropina metric with the navigation data $(\textbf{\textsf{h}},W)$ where $\textbf{\textsf{h}}$ is a left invariant Einstein Riemannian metric. Using the explicit computation in case $2$, $G$ admits such left invariant vector field $W$ if and only if $W=\dfrac{1}{\sqrt{\nu}}z$. We can see in the cases $4$ and $5$, $G$ does not admit any left invariant Killing vector field.}
\end{proof}

\section{Einstein Kropina Metrics on Homogeneous Spaces}
In this section we study invariant Einstein Kropina metrics on homogeneous spaces. Consider a homogeneous space $M=G/H$ with reductive decomposition $\mathfrak{g}=\mathfrak{h}\oplus \mathfrak{m}$. Let $T_o(G/H)$ be the tangent space on $M$ in the origin. The tangent space $T_o(G/H)$ can be identified with $\mathfrak{m}$ through the following map
\begin{center}
$X\longmapsto \dfrac{d}{dt}(\exp(tX)\cdot H)\mid _{t=0}$.
\end{center}
Note that for any $X\in \mathfrak{g}$ we can define a vector field $X^*$ on $G/H$ which is defined by the following equation,
\begin{center}
$X^*_{gH}= \dfrac{d}{dt}(\exp(tX)\cdot gH)\mid _{t=0}$.
\end{center}
The following result is a generalization of Theorem \ref{Theorem 1}. This theorem provides a way for constructing Einstein Kropina metrics on homogeneous spaces.\\

From now on, for a homogeneous Riemannian space $G/H$, we suppose that $\frak{m}$ is the orthogonal complement of $\frak{h}$ in $\frak{g}$ with respect to the inner product induced by the Riemannian metric. Also a vector $X\in\frak{m}$ is called $Ad(H)$-invariant if for any $h\in H$, $Ad(h)X=X$.

\begin{theorem}\label{Theorem 4}
Let $G/H$ be an $n$-dimensional homogeneous Riemannian space equipped with an invariant Einstein Riemannian metric $\textbf{\textsf{h}}$, with the Einstein scalar $\sigma$. Assume that $\tilde{W}$ is the vector field on $G/H$ generated by an $Ad(H)$-invariant vector $W\in \mathfrak{m}$ such that $\Vert \tilde{W} \Vert _{\textbf{\textsf{h}}}=1$ and $W$ satisfies
$$<[W,W_1],W_2>+<W_1,[W,W_2]>=0 \qquad \forall \,\, W_1,W_2 \in \mathfrak{m}.$$
 Then, the Kropina metric corresponded to the navigation data $(\textbf{\textsf{h}},\tilde{W})$ is a non-Riemannian Einstein homogeneous Finsler metric with the same Einstein scalar $\sigma$. Moreover, if $n \geq 3$ then the corresponding Kropina metric is Ricci constant.
\end{theorem}
\begin{proof}
According to proposition 7.7 of \cite{Deng},  $\tilde{W}$ is a killing vector field on $G/H$. So, a similar argument as in the proof of Theorem \ref{Theorem 1} completes the proof.
\end{proof}

\begin{remark}\label{The remark of Theorem 4}
 It is well-known that for any $W\in \mathfrak{m}$ the fundamental vector field $W^*$ is a killing vector field on $G/H$ (see \cite{Arvanitoyeorgos}). If it is also an invariant vector field then, the Kropina metric corresponded to the navigation data $(\textbf{\textsf{h}},W^*)$ is a non-Riemannian Einstein homogeneous Finsler metric. In particular, if $W\in Z(\mathfrak{g})\cap \mathfrak{m}$ then the fundamental vector field $W^*$ generated by $W$ is an invariant vector field.
\end{remark}

Suppose that $M=G/H$ is a compact homogeneous space with reductive decomposition $\mathfrak{g}=\mathfrak{h}\oplus \mathfrak{m}$.
It was shown that all invariant Riemannian metrics on $M$ can be described as follows.
The adjoint action of $H$ on $\mathfrak{m}$ splits $\mathfrak{m}$ as
\begin{equation}\label{split}
    \mathfrak{m}= \mathfrak{m}_0 \oplus \mathfrak{m}_1 \oplus ...\oplus \mathfrak{m}_r,
\end{equation}
where $H$ acts irreducibly on $\mathfrak{m}_1,\cdots,\mathfrak{m}_r$ and the action on $\mathfrak{m}_0 $ is trivial. Assume that the isotropy representations on $\mathfrak{m}_1,\cdots,\mathfrak{m}_r$ are inequivalent and $\textbf{B}$ is a bi-invariant Riemannian metric on $G$. Then, any invariant Riemannian metric is of the form
\begin{equation}\label{invariant metric on homogeneous spaces}
  \langle,\rangle =\textbf{\textsf{a}}\mid _{\mathfrak{m}_0 } + \sum _{i=1}^r \alpha _i \textbf{B}\mid _{\mathfrak{m}_i},
\end{equation}
where $\textbf{\textsf{a}}$ is an arbitrary metric on $\mathfrak{m}_0$ and $\alpha _i>0$ (see\cite{Ziller}).\\

Based on the preceding results and using the above description of invariant Riemannian metrics, W. Ziller has classified the homogeneous Einstein metrics on compact symmetric spaces of rank one (see \cite{Ziller}). Now we use the classifications given in \cite{Ziller} to describe some Einstein Kropina metrics on the spheres and projective spaces.
\begin{remark}
Montgomery-Samelson and Borel classified all the compact Lie groups acting transitively on a sphere and Ziller summarized these results in the following table (see \cite{Ziller}).\\

\begin{tabular}{ | ccccc | }\hline     \label{table}
&Spheres & $G$ & $H$ & Isotropy representation \\\hline
(1)&$S^n$ & $SO(n+1)$ & $SO(n)$ & Irreducible \\\hline
(2)&$S^{2n+1}$ & $SU(n+1)$ & $SU(n)$ & $\mathfrak{m} =\mathfrak{m}_0\oplus \mathfrak{m}_1$ \\\hline
(3)&$S^{2n+1}$ & $U(n+1)$ & $U(n)$ & $\mathfrak{m} =\mathfrak{m}_0\oplus \mathfrak{m}_1$ \\\hline
(4)&$S^{4n+3}$ & $Sp(n+1)$ & $Sp(n)$ & $\mathfrak{m} =\mathfrak{m}_0\oplus \mathfrak{m}_1$ \\\hline
(5)&$S^{4n+3}$ & $Sp(n+1)\times Sp(1)$ & $Sp(n)\times Sp(1)$ & $\mathfrak{m} =\mathfrak{m}_1\oplus \mathfrak{m}_2$  \\\hline
(6)&$S^{4n+3}$ & $Sp(n+1)\times U(1)$ & $Sp(n)\times U(1)$ & $\mathfrak{m} =\mathfrak{m}_0\oplus \mathfrak{m}_1\oplus \mathfrak{m}_2$ \\\hline
(7)&$S^{15}$ & $Spin(9)$ & $Spin(7)$ & $\mathfrak{m} =\mathfrak{m}_1\oplus \mathfrak{m}_2$ \\\hline
(8)&$S^7$ &$Spin(7)$ & $G_2$ & Irreducible \\\hline
(9)&$S^6$ & $G_2$ & $SU(3)$ & Irreducible \\\hline
\end{tabular} 
\end{remark}
\begin{theorem}
Suppose that $F$ is a non-Riemannian Kropina metric with the navigation data $(\textbf{\textsf{h}},W)$ on a sphere $S^n$, where $W$ is an invariant Killing vector field of the Riemannian manifold $(S^n,\textbf{\textsf{h}})$ such that $\Vert W\Vert _{\textbf{\textsf{h}}}=1$, and ${\textbf{\textsf{h}}}$ is an invariant Riemannian metric. Then, $F$ is an Einstein metric if and only if $(S^n,\textbf{\textsf{h}})$ is one of the following cases,\\
\begin{enumerate}
\item[(i)] $S^n=S^{2m+1}$ with  $m$ even ($m\geq 2$) and the Einstein Riemannian metric $\textbf{\textsf{h}}$ is the standard metric.
\item[(ii)] $S^n=S^{4m+3}$  ($m\geq 3$) and $\textbf{\textsf{h}}$ is  a $Sp(m + 1)$-invariant Einstein Riemannian metric.
\end{enumerate}

\end{theorem}
\begin{proof}
Using \cite{Zhang-Shen}, the existence  of  Einstein Kropina metrics on spheres is equivalent to the existence of  Einstein Riemannian metrics and invariant Killing vector fields. According to lemma 7.3 and theorem 7.7 of \cite{Deng}, such a Riemannian metric and such a vector field, only exist in cases 2,3,4, and 6 of the previous table and it completes the proof.
\end{proof}

Finally, we turn our attention to the projective spaces.
\begin{theorem}
 Projective spaces, $P^n\mathbb{C}=SU(n+1)/S(U(1)\times U(n))$, $P^n\mathbb{H}=Sp(n+1)/Sp(n)\times Sp(1)$, $P^2\mathbb{C}a=F_4/Spin(9)$ and $P^{2n+1}\mathbb{C}=Sp(n+1)/Sp(n)\times U(1)$ do not admit any  non-Riemannian Einstein Kropina metrics.
\end{theorem}
\begin{proof}
By \cite{Ziller}, in all cases, there is no nonzero invariant vector field on projective spaces.
\end{proof}
{\large{\textbf{Acknowledgment.}}} We are grateful to the office of Graduate Studies of the University of Isfahan for their support.


\begin{thebibliography}{99}

\bibitem{Antonelli-Ingarden-Matsumoto} P. L. Antonelli, R. S. Ingarden, M. Matsumoto, \emph{The Theory of Sprays and Finsler Spaces with Applications in Physics and Biology}, Kluwer, Dordrecht (1993).

\bibitem{Arvanitoyeorgos} A. Arvanitoyeorgos, \emph{An Introduction to Lie Groups and the Geometry of Homogeneous Spaces}, AMS, Student Mathematical Library, Vol. 22, (2003).

\bibitem{Asanov} G. S. Asanov, \emph{Finsler Geometry, Relativity and Gauge Theories}, D. Reidel, Dordrecht,
(1985).

\bibitem{Chern-Shen} S. S. Chern and Z. Shen, \emph{Riemann-Finsler Geometry}, World Scientific, Singapore, (2005).

 \bibitem{Deng} S. Deng, \emph{Homogeneous Finsler Spaces}, Springer, New York, (2012).

\bibitem{Salimi Moghaddam} H. R. Salimi Moghaddam, \emph{Left invariant Ricci solitons on three-dimensional Lie groups}, J. Lie Theory, \textbf{29(4)}(2019), 957-968.

\bibitem{Torres-Book} G. F. Torres del Castillo, \emph{Differentiable Manifolds, A Theoretical Physics Approach}, Birkh\"{a}user, Basel, (2012).

\bibitem{Xia} Q. Xia, \emph{On Kropina metrics of scalar flag curvature}, Differential Geom. Appl., \textbf{31}(2013) 393-404.

\bibitem{Yoshikawa-Sabau} R. Yoshikawa and S. V. Sabau, \emph{Kropina metrics and Zermelo navigation on Riemannian manifolds}, Geom. Dedicata, \textbf{171}(2014), 119-148.

\bibitem{Zhang-Shen} X. Zhang and Y-B Shen, \emph{On Einstein-Kropina metrics}, Differential Geom. Appl., \textbf{31}(2013) 80-92.

\bibitem{Ziller} W. Ziller, \emph{Homogeneous Einstein metrics on spheres and projective spaces}, Math. Ann., \textbf{259}(1982) 351-358.
\end{thebibliography}
\end{document}